\documentclass{amsart}

\usepackage{amssymb}
\usepackage{mathrsfs}

\usepackage{color}

\usepackage[colorlinks=true, linkcolor=red, urlcolor=magenta, citecolor=blue]{hyperref}

\usepackage{enumitem}

\usepackage{comment}
\usepackage[all]{xy}

\usepackage[alphabetic]{amsrefs}

\newtheorem{theorem}{Theorem}
\newtheorem{proposition}[theorem]{Proposition}
\newtheorem{cor}[theorem]{Corollary}

\theoremstyle{definition}
\newtheorem{remark}[theorem]{Remark}

\newcommand{\cM}{\mathcal{M}}
\newcommand\cO{\mathcal{O}}
\newcommand\cT{\mathscr{T}}
\renewcommand\AA{\mathbb{A}}
\newcommand\CC{\mathbb{C}}
\newcommand\NN{\mathbb{N}}
\newcommand\PP{\mathbb{P}}
\newcommand\bP{\mathbb{P}}
\newcommand\bG{\mathbb{G}}
\newcommand\QQ{\mathbb{Q}}
\newcommand\RR{\mathbb{R}}
\newcommand\ZZ{\mathbb{Z}}

\DeclareMathOperator{\lct}{lct}
\newcommand{\GIT}{\mathrm{GIT}}
\newcommand{\K}{\mathrm{K}}

\newcommand{\onto}{\twoheadrightarrow} 
\newcommand{\toric}[1]{\mathrm{TV}(#1)} 

\newcommand{\cTqG}[2]{\mathscr{T}^{\mathrm{qG}, #1}_{#2}}
\newcommand{\DefqG}[1]{\mathrm{Def}^{\mathrm{qG}}(#1)}
\DeclareMathOperator{\Aut}{Aut}
\newcommand{\Gm}{\mathbb{G}_\mathrm{m}}
\DeclareMathOperator{\Proj}{Proj} 
\DeclareMathOperator{\Hom}{Hom} 
\newcommand{\bQ}{\mathbb{Q}}
\newcommand{\bfA}{\mathbf{A}}
\newcommand{\cX}{\mathcal{X}}
\newcommand{\SL}{\mathrm{SL}}
\newcommand{\bfP}{\mathbf{P}}
\newcommand{\bZ}{\mathbb{Z}}
\newcommand{\Pic}{\mathrm{Pic}}
\newcommand{\sslash}{\mathbin{/\mkern-6mu/}}

\title[On K-stability of some del Pezzo surfaces of Fano index 2]{On K-stability of some del Pezzo surfaces \\ of Fano index 2}

\author{Yuchen Liu}
\address{Department of Mathematics, Northwestern University, Evanston, IL 60208, USA}
\email{yuchenl@northwestern.edu}

\author{Andrea Petracci}
\address{Institut f\"ur Mathematik, Freie Universit\"at Berlin, Arnimallee 3, Berlin 14195, Germany}
\email{andrea.petracci@fu-berlin.de}

\begin{document}

\begin{abstract}
For every integer $a \geq 2$, we relate the K-stability of hypersurfaces in the weighted projective space $\PP(1,1,a,a)$ of degree $2a$ with the GIT stability of binary forms of degree $2a$.
Moreover, we prove that such a hypersurface is K-polystable and not K-stable if it is quasi-smooth.
\end{abstract}

\maketitle

\section{Introduction}

It is an important problem in algebraic geometry and in differential geometry to decide if a given Fano variety $X$ admits a K\"ahler--Einstein (KE) metric. The Yau--Tian--Donaldson (YTD) Conjecture predicts that the existence of a KE metric on $X$ is equivalent to the K-polystability of $X$.
Using Cheeger--Colding--Tian theory, the YTD Conjecture was first proved when $X$ is smooth \cite{chen_donaldson_sun, tian, berman_polystability}, when $X$ is $\bQ$-Gorenstein smoothable \cite{LWX19, SSY16}, or when $X$ has dimension $2$ \cite{ltw_1}. Later, a different method, namely the variational approach, was introduced in \cite{BBJ15}. The analytic side of the variational approach was completed in \cite{ltw_2, Li19} which shows that a $\bQ$-Fano variety $X$, that is, a Fano variety with klt singularities, admits a KE metric if and only if $X$ is reduced uniformly K-stable, a concept introduced in \cite{His16} as an equivariant version of uniform K-stability (see also \cite{XZ19}). Recently, using purely algebro-geometric methods, the work \cite{LXZ21} establishes the equivalence between K-polystability and reduced uniform K-stability. This work, combining with the variational approach, proves the YTD Conjecture for all $\bQ$-Fano varieties.



K-stability of del Pezzo surfaces which are quasi-smooth hypersurfaces in weighted projective $3$-spaces has been studied extensively.
Johnson and Koll\'ar \cite{kollar_johnson} classified those which are anticanonically polarised (i.e.\ have Fano index $1$) and decided the existence of a KE metric on many of these, by using Tian's criterion which relates KE metrics to global log canonical thresholds (also called $\alpha$-invariants) \cite{tian_alpha, nadel_annals, demailly_kollar, cheltsov_lct, odaka_sano_alpha, Fuj19b}.
This method was applied to most of these del Pezzo surfaces by Araujo~\cite{araujo_del_pezzo}, Boyer--Galicki--Nakamaye~\cite{boyer_galicki_nakamaye}, and Cheltsov--Park--Shramov~\cite{cheltsov_park_shramov_exceptional}.
One case was missing and was finally solved in \cite{cheltsov_park_shramov_delta} by using delta invariants
(see \cite{fujita_odaka_delta, blum_jonsson}).

The (non-)existence of KE metrics on many del Pezzo surfaces which are quasi-smooth hypersurfaces in weighted projective $3$-spaces with Fano index $\geq 2$ has been studied in \cite{cheltsov_park_shramov_exceptional, cheltsov_park_shramov_delta, del_pezzo_zoo, kim_won}.

In this paper, we study K-polystability of quasi-smooth degree $2a$ hypersurfaces in the weighted projective space $\PP(1,1,a,a)$. When $a\in \{2,4\}$, such del Pezzo surfaces are $\bQ$-Gorenstein smoothable, and their K-polystability was determined by Mabuchi--Mukai \cite{MM93} and Odaka--Spotti--Sun \cite{odaka_spotti_sun} (see Remark \ref{rem:a=2or4}). To the authors' knowledge it is not known if they are K-polystable for an integer $a=3$ or $a \geq 5$.
In \cite{kim_won} Kim and Won conjecture that these surfaces are K-polystable and not K-stable.

Our main result relates the K-polystability (resp.\ K-semistability) of degree $2a$ hypersurfaces in $\PP(1,1,a,a)$ to GIT polystability (resp.\ GIT semistability) of degree $2a$ binary forms (see \cite[Chapter~4]{git}).

\begin{theorem} \label{thm:k=git}
	Let $a \geq 2$ be an integer and 
	let $\bP(1,1,a,a)$ be the weighted projective space with coordinates $[x, y, z, w]$ with weights $\deg x = \deg y = 1$ and $\deg z = \deg w = a$. Let $X$ be a hypersurface of degree $2a$ in $\bP(1,1,a,a)$.
	
	Then $X$ is K-semistable (resp.\ K-polystable) if and only if, after an automorphism of $\bP(1,1,a,a)$, the equation of $X$ is given by $z^2+w^2+g(x,y)=0$ where $g\neq 0$ is GIT semistable (resp.\ GIT polystable) as a degree $2a$ binary form.
	Moreover, $X$ is not K-stable.
\end{theorem}

As a consequence we prove the K-polystability of quasi-smooth hypersurfaces in $\PP(1,1,a,a)$ of degree $2a$, hence partially confirming \cite[Conjecture~1.3]{kim_won}.

\begin{cor} \label{cor:k-poly}
Let $a \geq 2$ be an integer and let $X$ be a degree $2a$ quasi-smooth hypersurface in $\PP(1,1,a,a)$. Then $X$ is K-polystable and not K-stable. Moreover, $X$ admits a KE metric.
\end{cor}

Recently the result of this corollary has been independently  announced by Viswanathan using different methods.

It is possible to give a proof of K-polystability for a general hypersurface in $\PP(1,1,a,a)$ of degree $2a$, when $a$ is odd, by analysing the deformation theory of the toric surface appearing in Proposition~\ref{prop:toric} similarly to \cite{ask_petracci} and without using Theorem~\ref{thm:k=git}.

\subsection*{Notation and conventions}
We always work over $\CC$.
A \emph{del Pezzo surface} is a normal projective surface whose anticanonical divisor is $\QQ$-Cartier and ample.
Every toric variety we consider is normal.
We do not even try to write down the definitions of K-(poly/semi)stability of Fano varieties and of log Fano pairs:
we refer the reader to the excellent survey \cite{xu_survey}, the paper \cite{ADL19}, and to the references therein.

\subsection*{Acknowledgements}
The second author wishes to thank Anne-Sophie Kaloghiros for many fruitful conversations and Yuji Odaka for helpful e-mail exchanges;
he is grateful also to Ivan Cheltsov and Jihun Park for useful remarks on an earlier draft of this manuscript and for sharing a preliminary version of \cite{kim_won}. The first author is partially supported by the NSF Grant DMS-2001317.

\section{Proofs}

In what follows $a$ is a fixed integer greater than $1$.
We consider the weighted projective space $\PP(1,1,a,a)$ with coordinates $[x, y, z, w]$ with weights $\deg x = \deg y = 1$ and $\deg z = \deg w = a$.

\begin{proposition} \label{prop:toric}
If $Y$ is the hypersurface in $\PP(1,1,a,a)$ defined by the equation $zw - x^a y^a  = 0$, then $Y$ is a K-polystable toric del Pezzo surface.	
\end{proposition}

\begin{proof}
	We fix the lattice $N = \ZZ^2$ and its dual $M = \Hom_\ZZ(N,\ZZ)$. Elements of $N$ will be columns and elements of $M$ will be rows. 
	
	Let $Q$ be the convex hull of the points 
	\begin{equation*}
	(0,0), \ (0,1), \ (a^{-1},0), \ ( -a^{-1},1)
	\end{equation*}
     in $M_\RR$.
	Let $\Sigma$ be the inner normal fan of $Q$; thus $\Sigma$ is the complete normal fan in $N$ whose rays are generated by the vectors
	\begin{equation} \label{eq:vertices_of_P}
	\begin{pmatrix}
	a \\ 1
	\end{pmatrix},
	\
	\begin{pmatrix}
	0 \\ 1
	\end{pmatrix},
	\
	\begin{pmatrix}
	-a \\ -1
	\end{pmatrix},
	\
	\begin{pmatrix}
	0 \\ -1
	\end{pmatrix}.
	\end{equation}
	We want to show that $Y$ is the toric variety associated to the fan $\Sigma$.
	
	Provisionally, let $\toric{\Sigma}$ denote the toric variety associated to $\Sigma$.
	Consider the cone $\tau$ in $M \oplus \ZZ$ spanned by $Q \times \{1\}$.
	Consider the finitely generated monoid $\tau \cap (M \oplus \ZZ)$ and the semigroup algebra $\CC[\tau \cap (M \oplus \ZZ)]$, 
	which is $\NN$-graded via the projection $M \oplus \ZZ \onto \ZZ$.
	Toric geometry says that $\toric{\Sigma} = \Proj \CC[\tau \cap (M \oplus \ZZ)]$.
	One can see that the minimal set of generators of the semigroup $\tau \cap (M \oplus \ZZ)$ is made up of the vectors
	\begin{equation*}
	(0,0,1), \ (0,1,1), \ (1,0,a), \ (-1,a,a);
	\end{equation*}
	these vectors satisfy a unique relation:
	\[
	a(0,0,1) + a(0,1,1) = (1,0,a) + (-1,a,a).
	\]
	Hence the $\NN$-graded ring $\CC[\tau \cap (M \oplus \ZZ)]$ coincides with $\CC[x,y,z,w] / (zw - x^a y^a)$, where $\deg x = \deg y = 1$ and $\deg z = \deg w = a$. Therefore $Y = \toric{\Sigma}$.
	
	The vectors in \eqref{eq:vertices_of_P} are the vertices of a polytope $P$ in $N$. This implies that $Y$ is a del Pezzo surface, i.e.\ $-K_Y$ is $\QQ$-Cartier and ample.
	
	Let $P^\circ$ be the polar of $P$; thus $P^\circ$ is the convex hull of $(0, \pm 1)$ and $\pm (\frac{2}{a}, -1)$ in $M_\RR$. 
	The polygon $P^\circ$ is the moment polytope of the toric boundary of $Y$, which is an anticanonical divisor.
	Since $P$ is centrally symmetric, also $P^\circ$ is centrally symmetric, thus the barycentre of $P^\circ$ is the origin. By \cite{berman_polystability} $Y$ is K-polystable.
\end{proof}

\begin{remark}
\begin{enumerate}
 \item Another way to show K-polystability of $Y$ is by realising $Y\cong (\bP^1\times\bP^1)/(\bZ/a\bZ)$, where the $\bZ/a\bZ$-action on $\bP^1\times\bP^1$ is given by 
 \[
  \zeta\cdot ([u_0,u_1],[v_0,v_1]):=([\zeta u_0, u_1],[\zeta^{-1} v_0,v_1])\quad \textrm{with }\zeta=e^{\frac{2\pi i}{a}}.
 \]
 Since the above action is free away from finitely many points, and it preserves the product of Fubini-Study metrics on $\bP^1\times\bP^1$, we know that $Y$ admits a KE metric and hence is K-polystable by \cite{berman_polystability}.

 \item A degree $2a$ hypersurface in $\PP(1,1,a,a)$ is defined by an equation
	\[
	q(z,w) + f(x,y) z + h(x,y) w + g(x,y) = 0
	\]
	where $q$ is a quadratic form, $f$ and $h$ are forms of degree $a$, and $g$ is a form of degree $2a$.
	With an automorphism of $\PP(1,1,a,a)$ which is induced by a linear change of the coordinates $z,w$, we can diagonalise the quadratic form $q$, so that the term $zw$ disappears.
	Furthermore, if $q$ has full rank, with an automorphism of $\PP(1,1,a,a)$ induced by $z \mapsto z + \frac{f}{2}$ and $w \mapsto w + \frac{h}{2}$, the equation becomes
	\[
	z^2 + w^2 + g(x,y) = 0.
	\]
\end{enumerate}
\end{remark}

\begin{proof}[Proof of Theorem~\ref{thm:k=git}]
	We start from the ``if'' part. 
	Suppose $X\subset\bP(1,1,a,a)$ is defined by the equation $z^2+w^2+g(x,y)=0$ with $g\neq 0$. Then the ``if'' part states that $X$ is K-semistable (resp.\ K-polystable) if $g$ is GIT semistable (resp.\ GIT polystable).
	
	By forgetting the $w$-coordinate, we obtain a double cover $\pi: X\to \bP(1,1,a)$ with branch locus $D=(z^2+g(x,y)=0)$. Thus by \cite{LZ20, Zhu20} we know that $X$ is K-semistable (resp.\ K-polystable) if and only if $(\bP(1,1,a),\frac{1}{2}D)$ is K-semistable (resp.\ K-polystable).
	
	Let us assume for the moment that $g$ is an arbitrary degree $2a$ binary form. Denote by $D_0:=(z^2=0)$ as a divisor on $\bP(1,1,a)$.  It is clear that $\bP(1,1,a)$ is the projective cone over $\bP^1$ with polarization $\cO_{\bP^1}(a)$, and $\frac{1}{2}D_0$ is the section at infinity. Since $\bP^1$ is K\"ahler--Einstein,  \cite[Proposition 3.3]{LL19} shows that $(\bP(1,1,a), (1-\frac{r}{2})\frac{1}{2}D_0)$ admits a conical KE metric, where $r\in \bQ_{>0}$ satisfies $\cO_{\bP^1}(a)\sim_{\bQ} -r^{-1} K_{\bP^1}$, i.e. $r=\frac{2}{a}$. By computation, $(1-\frac{r}{2})\frac{1}{2}=\frac{a-1}{2a}$. Thus $(\bP(1,1,a), \frac{a-1}{2a}D_0)$ admits a conical KE metric and hence is K-polystable.
	It is clear that under the $\bG_m$-action $\sigma$ on $\bP(1,1,a)$ given by $\sigma(t)\cdot [x,y,z]=[x,y,tz]$, the log Fano pair $(\bP(1,1,a), \frac{a-1}{2a}D)$ specially degenerates to $(\bP(1,1,a), \frac{a-1}{2a}D_0)$ as $t\to 0$. Thus by openness of K-semistability \cite{BLX19, Xu19} we know that $(\bP(1,1,a), \frac{a-1}{2a}D)$ is K-semistable. 
	
	Next, we assume that $g\neq 0$ is GIT semistable. By GIT of binary forms, we know that each linear factor in $g(x,y)$ has multiplicity at most $a$. In other words, the curve $D$ has only $A_{k-1}$-singularities (i.e.\ locally analytically given by $x^2+y^k=0$) where $k\leq a$. Thus we have that $\lct(\bP(1,1,a);D)\geq \frac{1}{2}+\frac{1}{a}=\frac{a+2}{2a}$. This implies that $(\bP(1,1,a), \frac{a+2}{2a}D)$ is a log canonical log Calabi--Yau pair. Thus interpolation for K-stability \cite[Proposition 2.13]{ADL19} implies that $(\bP(1,1,a), \frac{1}{2}D)$ is K-semistable. 
	
	Next, we assume that $g\neq 0$ is GIT polystable. There are two cases: $g$ is strictly GIT polystable (i.e.\ GIT polystable but not GIT stable), or $g$ is GIT stable. In the first case, under a suitable coordinate we may write $g(x,y)=x^a y^a$. Thus the double cover $X$ is toric, and as shown in Proposition~\ref{prop:toric} $X$ is K-polystable. In the second case, we know that each linear factor in $g(x,y)$ has multiplicity at most $a-1$. Thus the curve $D$ has only $A_{k-1}$-singularities where $k\leq a-1$. Thus we have that $\lct(\bP(1,1,a);D)\geq \frac{1}{2}+\frac{1}{a-1}>\frac{a+2}{2a}$, which implies that $(\bP(1,1,a), \frac{a+2}{2a}D)$ is a klt log Calabi--Yau pair. Thus interpolation for K-stability \cite[Proposition 2.13]{ADL19} implies that $(\bP(1,1,a), \frac{1}{2}D)$ is K-stable. This finishes the proof of the ``if'' part.
	
	Next, we treat the ``only if'' part. In fact, this follows from moduli comparison arguments as in \cite{ADL19}. Let $\bfA:=H^0(\bP^1, \cO_{\bP^1}(2a))$ be the affine space parametrizing degree $2a$ binary forms.
	Let $\bfA^{\rm ss}\subset \bfA\setminus\{0\}$ be the open subset of GIT semistable binary forms.
	Consider the universal family of weighted hypersurfaces $\cX\to \bfA^{\rm ss}$ where $\cX\subset \bP(1,1,a,a)\times \bfA^{\rm ss}$ has fibre $(z^2+w^2+g(x,y)=0)$ over each $g\in \bfA^{\rm ss}$.
	By the ``if'' part we know that each fibre of $\cX\to \bfA^{\rm ss}$ is K-semistable. 
	Consider the $(\bG_m\times \SL_2)$-action $\lambda$ on $\bfA$ given by 	$\lambda(t,A)\cdot g(x,y)=t^2 g(A^{-1}(x,y))$. It is clear that $\bfA^{\rm ss}$ is a $(\bG_m\times\SL_2)$-invariant open subset. Then there is a $(\bG_m\times\SL_2)$-action $\tilde{\lambda}$ on $\cX$ as a lifting of $\lambda$ given by 
	\[
	 \tilde{\lambda}(t,A)\cdot ([x,y,z,w],g):=([A(x,y),tz,tw], \lambda(t,A)\cdot g).
	\]
    Denote by $\cM^{\GIT}:=[\bfA^{\rm ss}/(\bG_m\times\SL_2)]$ and $M^{\GIT}:=\bfP\sslash \SL_2$ where $\bfP:=\bP(\bfA)$. It is clear that $M^{\GIT}$ is the good moduli space of $\cM^{\GIT}$. Taking quotient of the family $\cX\to\bfA^{\rm ss}$ by $
    \tilde{\lambda}$, we obtain a $\bQ$-Gorenstein flat family of K-semistable $\bQ$-Fano varieties over $\cM^{\GIT}$, where fibres over closed points are precisely K-polystable fibres. 
    
    From a series of important recent works \cite{Jia17, LWX18, CP18, blum_xu_uniqueness, ABHLX19, Xu19, BLX19, XZ19, XZ20, BHLLX20, LXZ21}, we know that there exists an Artin stack of finite type $\cM_{2,8/a}^{\rm Kss}$ parametrizing K-semistable (possibly singular) del Pezzo surfaces of degree $8/a$. Moreover, $\cM_{2,8/a}^{\rm Kss}$ admits a projective good moduli space $M_{2,8/a}^{\rm Kps}$ parametrizing K-polystable ones. Let $\cM^{\K}$ be the Zariski closure (with reduced structure) of the locally closed substack in $\cM_{2,8/a}^{\rm Kss}$ parametrizing K-semistable degree $2a$ weighted hypersurfaces $X\subset \bP(1,1,a,a)$. Let $M^{\K}$ be the good moduli space of $\cM^{\K}$ as a closed  algebraic subspace of $M_{2,8/a}^{\rm Kps}$. 
    Then the above construction and the ``if'' part produces a morphism $\Phi:\cM^{\GIT}\to \cM^{\K}$ which descends to a morphism $\phi:M^{\GIT}\to M^{\K}$. Since a general weighted hypersurface $X$ has the form $z^2+w^2+g(x,y)=0$ in a suitable coordinate where $g\neq 0$ has no multiple linear factors, we know that $\Phi$ is dominant. The ``if'' part shows that $\Phi$ sends closed points to closed points. Since $M^{\GIT}$ is projective, we know that $\phi$ is proper and dominant, which implies that $\phi$ is surjective. Moreover, since $\SL_2$ has no non-trivial characters, we have injections
    \[
     \Pic(M^{\GIT})=\Pic(\bfP\sslash\SL_2)\hookrightarrow \Pic_{\SL_2}(\bfP^{\rm ss})\hookrightarrow \Pic(\bfP^{\rm ss})
    \]
    by \cite[Proposition 4.2 and Section 2.1]{KKV89}. It is clear that $\bfP\setminus\bfP^{\rm ss}$ has codimension at least $2$ in $\bfP$. Thus we have $\Pic(\bfP^{\rm ss})\cong \Pic(\bfP)\cong \bZ$. In particular, the GIT quotient $M^{\GIT}$ has Picard rank $1$. It is clear that $M^{\K}$ is not a single point. Thus $\phi:M^{\GIT}\to M^{\K}$ is a finite surjective morphism by Zariski's main theorem.
    
    Next, we show that K-poly/semistability implies GIT poly/semistability. Since $\phi$ is surjective,  a K-polystable hypersurface $X\subset\bP(1,1,a,a)$ satisfies that $[X]=\phi([g])\in M^{\K}$ for some GIT polystable binary form $g\in \bfA\setminus\{0\}$. Thus $X$ has the form $z^2+w^2+g(x,y)=0$ with $g\neq 0$ being GIT polystable. If $X\subset\bP(1,1,a,a)$ is K-semistable, then it specially degenerates to a K-polystable point $[X_0]\in M^{\K}$ by \cite{LWX18}. Clearly $X_0$ has the form $z^2+w^2+g_0(x,y)=0$ with $g_0\neq 0$ being GIT polystable. Since the rank of quadratic forms cannot jump up under degeneration, the quadratic terms in $(z,w)$ of the equation of $X$ has rank $2$, which implies that $X=(z^2+w^2+g(x,y)=0)$ for some $g$. By \cite[Corollary 1.7]{Fuj19}, we know that $(\bP(1,1,a), \frac{1}{2}D)$ is K-semistable where $D=(z^2+g(x,y)=0)$. Since $X$ carries a $\bZ/2\bZ$-action given by $w\mapsto -w$, we may assume that the special degeneration from $X$ to $X_0$ is $\bZ/2\bZ$-equivariant by \cite{LZ20, Zhu20}. In particular, this shows that $(\bP(1,1,a), \frac{1}{2}D)$ specially degenerates to $(\bP(1,1,a), \frac{1}{2}D_0)$ where $D_0=(z^2+g_0(x,y)=0)$. By the lower semi-continuity of lct (see e.g.\ \cite{demailly_kollar}), we know that $\lct(\bP(1,1,a); D)\geq \lct(\bP(1,1,a);D_0)\geq \frac{a+2}{2a}$ where the latter inequality was proven in the ``if'' part due to the fact that $g_0$ is GIT polystable. Thus this shows that $g\neq 0$, and each linear factor in $g(x,y)$ has multiplicity at most $a$. Thus we obtain the GIT semistability of $g$.
    The proof of the ``only if'' part is finished.
   
	Finally, we show that any hypersurface $X\subset\bP(1,1,a,a)$ of degree $2a$ is not K-stable. If $X$ were K-stable, then it would have equation $z^2 + w^2 + g(x,y) = 0$, or equivalently the equation $zw + g(x,y) = 0$. It is clear that $t \cdot (z,w) = (tz, t^{-1}w)$ defines an effective action of $\Gm$ on $X$. Thus $X$ is not K-stable by definition.
\end{proof}

	\begin{proof}[Proof of Corollary~\ref{cor:k-poly}]		
		It is clear that $X$ is quasi-smooth if and only if, up to an automorphism of $\PP(1,1,a,a)$, $X$ has the equation $z^2+w^2+g(x,y)=0$ where $g$ has no multiple linear factors. Thus by Theorem~\ref{thm:k=git}  we conclude that $X$ is K-polystable and not K-stable. The existence of KE metrics on $X$ follows from \cite{ltw_1}.
	\end{proof}
	
	\begin{remark}\label{rem:a=2or4}
		For $a=2$, the del Pezzo surface $X$ admits an embedding into $\bP^4$ as a complete intersection of two hyperquadrics. This is induced by the linear system $|-K_X|$ which  is very ample.
		
		For $a=4$, $X$ (as a double cover of $\bP(1,1,4)$) appeared in \cite{odaka_spotti_sun} where it lies in the exceptional divisor of Kirwan blow-up of the GIT moduli space. Hence $X$ admits a $\bQ$-Gorenstein smoothing to degree $2$ smooth del Pezzo surfaces.
		
		Therefore, in both cases ($a=2$ or $a=4$) our K-moduli space $M^{\K}$, introduced in the proof of Theorem~\ref{thm:k=git}, form a divisor in the K-moduli spaces of $\bQ$-Gorenstein smoothable del Pezzo surfaces of degree $\frac{8}{a}$ studied in \cite{MM93, odaka_spotti_sun}.
		We will see in Proposition~\ref{prop:a=3or>=5} what happens for $a=3$ or $a \geq 5$. 
	\end{remark}

\begin{proposition} \label{prop:a=3or>=5}
	If $a=3$ or $a \geq 5$, then the locus of K-polystable degree $2a$ hypersurfaces in $\PP(1,1,a,a)$ is a connected component of
	$M_{2,8/a}^{\rm Kps}$.
\end{proposition}

\begin{proof}
We denote by $\Gamma$ the connected component of $M_{2,8/a}^{\rm Kps}$ containing K-polystable degree $2a$ hypersurfaces in $\PP(1,1,a,a)$.
In the proof of Theorem~\ref{thm:k=git} we showed that the locus of K-polystable degree $2a$ hypersurfaces in $\PP(1,1,a,a)$ is closed in $\Gamma$; this locus is denoted by $M^{\K}$.
We need to prove that  $M^{\K}$ coincides with $\Gamma$. We will achieve this by a dimension count. Using the notation of the proof of Theorem~\ref{thm:k=git}, there is a finite surjective morphism $\phi: M^{\GIT}\to M^{\K}$. Thus we have
\[
\dim M^{\K} = \dim M^{\GIT} = \dim \bfP - \dim \SL_2 = 2a-3.
\]

Let us now compute the dimension of $\Gamma$ by analysing the deformation theory of the K-polystable toric del Pezzo surface $Y$ introduced in Proposition~\ref{prop:toric}. Note that a similar study was discussed in \cite{MGS21}.

\medskip

Let $\cT^0_Y$ denote the sheaf of derivations on $Y$, i.e.\ the dual of $\Omega_Y^1$.
Let $\cTqG{1}{Y}$ denote the sheaf of $1$st order $\QQ$-Gorenstein deformations of $Y$.
The singular locus of $Y$, which consists of $4$ points, contains the set-theoretic support of $\cTqG{1}{Y}$.

Since $Y$ is a toric Fano, we have $H^1(\cT^0_Y) = H^2(\cT^0_Y) = 0$ by \cite[\S4.3]{petracci_survey}. Via a standard argument about the local-to-global spectral sequence for Ext, we deduce that the tangent space of the $\QQ$-Gorenstein deformation functor of $Y$ is $H^0(\cTqG{1}{Y})$.
The $\QQ$-Gorenstein deformation functor of $Y$ is unobstructed because $Y$ is a del Pezzo surface with cyclic quotient singularities \cite[Lemma~6]{procams}. Therefore the germ at the origin of the vector space $H^0(\cTqG{1}{Y})$ is the base of the miniversal (Kuranishi) $\QQ$-Gorenstein deformation of $Y$.

Consider the torus $T_N = N \otimes_\ZZ \Gm$ acting on the toric variety $Y$.
There is an action of $T_N$ on the vector space $H^0(\cTqG{1}{Y})$, hence $H^0(\cTqG{1}{Y})$ splits into the direct sum of irreducible representations (characters) of the torus $T_N$.

We observe that the singularities of $Y$ are:
\begin{itemize}
	\item $2$ points of type $\frac{1}{a}(1,-1) = A_{a-1}$, which correspond to the cones in $\Sigma$ spanned by
	\[
	\pm \begin{pmatrix}
	a \\ 1
	\end{pmatrix},
	\
	\pm \begin{pmatrix}
	0 \\ 1
	\end{pmatrix},
	\]
	\item $2$ points of type $\frac{1}{a}(1,1)$, which correspond to the cones in $\Sigma$ spanned by
	\[
	\pm \begin{pmatrix}
	a \\ 1
	\end{pmatrix},
	\
	\mp \begin{pmatrix}
	0 \\ 1
	\end{pmatrix}.
	\]
\end{itemize}
Since $a=3$ or $a \geq 5$, the surface singularity $\frac{1}{a}(1,1)$ is $\QQ$-Gorenstein rigid, so it does not contribute to $H^0(\cTqG{1}{Y})$. One can see that the $T_N$-representation $H^0(\cTqG{1}{Y})$ is the direct sum of the $1$-dimensional representation of $T_N$ associated to the characters
\begin{equation} \label{eq:weights}
(0, \pm 2), (0, \pm 3), \dots, (0, \pm a) \in M.
\end{equation}
In particular $\dim H^0(\cTqG{1}{Y}) = 2a-2$, so the base of the miniversal $\QQ$-Gorenstein deformation of $Y$ is a smooth germ of dimension $2a-2$.

Since the weights in \eqref{eq:weights} are contained in a rank $1$ sublattice of $M$, there exists a $1$-dimensional subtorus of $T_N$ which acts trivially on $H^0(\cTqG{1}{Y})$. More precisely one can prove that the affine quotient $H^0(\cTqG{1}{Y}) / T_N$ has dimension $2a-3$.

Since every facet of the polytope $P^\circ$ has no interior lattice points, by \cite[Proposition~2.6]{ask_petracci} the automorphism group of $Y$ is $T_N \rtimes \Aut(P)$, where $\Aut(P) \subseteq \mathrm{GL}(N)$ is the finite group consisting of the lattice automorphisms which keep the polytope $P$ invariant. Since the difference between $T_N$ and $\Aut(Y)$ is just a finite group, we deduce that the affine quotient the affine quotient $H^0(\cTqG{1}{Y}) / \Aut(Y)$ has dimension $2a-3$.
By the local structure of the K-moduli space \cite{ABHLX19, luna_etale_slice_stacks} we know that the completion of the local ring of $\Gamma$ at $[Y]$ coincides with the completion at the origin of $H^0(\cTqG{1}{Y}) / \Aut(Y)$. This proves that $\Gamma$ has dimension $2a-3$ at $[Y]$. Since $\dim M^{\K}=2a-3$, we know that $M^{\K}$ is an irreducible component of $\Gamma$. 

Moreover, since all K-polystable del Pezzo surfaces in $M^{\K}$ have cyclic quotient singularities by Theorem \ref{thm:k=git}, they have unobstructed $\bQ$-Gorenstein deformations by \cite[Lemma~6]{procams}. Thus the stack $\cM_{2,\frac{8}{a}}^{\rm Kss}$ is smooth in an open neighbourhood of $\cM^{\K}$. In particular, this implies that $\Gamma$ is normal in an open neighbourhood of $M^{\K}$. Since $M^{\K}$ is an irreducible component of $\Gamma$, we have $M^\K = \Gamma$.
\end{proof}

\bibliographystyle{alpha}
\bibliography{Biblio_final}

\begin{bibdiv}
\begin{biblist}

\bib{ABHLX19}{article}{
      author={Alper, Jarod},
      author={Blum, Harold},
      author={Halpern-Leistner, Daniel},
      author={Xu, Chenyang},
       title={Reductivity of the automorphism group of {K}-polystable {F}ano
  varieties},
        date={2020},
     journal={Invent. Math.},
      volume={222},
      number={3},
       pages={995\ndash 1032},
}

\bib{procams}{article}{
      author={Akhtar, Mohammad},
      author={Coates, Tom},
      author={Corti, Alessio},
      author={Heuberger, Liana},
      author={Kasprzyk, Alexander},
      author={Oneto, Alessandro},
      author={Petracci, Andrea},
      author={Prince, Thomas},
      author={Tveiten, Ketil},
       title={Mirror symmetry and the classification of orbifold del {P}ezzo
  surfaces},
        date={2016},
     journal={Proc. Amer. Math. Soc.},
      volume={144},
      number={2},
       pages={513\ndash 527},
}

\bib{ADL19}{article}{
      author={Ascher, Kenneth},
      author={DeVleming, Kristin},
      author={Liu, Yuchen},
       title={Wall crossing for {K}-moduli spaces of plane curves},
        date={2019},
  note={\href{https://arxiv.org/abs/1909.04576}{\textsf{arXiv:1909.04576}}},
}

\bib{luna_etale_slice_stacks}{article}{
      author={Alper, Jarod},
      author={Hall, Jack},
      author={Rydh, David},
       title={A {L}una \'{e}tale slice theorem for algebraic stacks},
        date={2020},
     journal={Ann. of Math. (2)},
      volume={191},
      number={3},
       pages={675\ndash 738},
}

\bib{araujo_del_pezzo}{article}{
      author={Araujo, Carolina},
       title={K\"{a}hler-{E}instein metrics for some quasi-smooth log del
  {P}ezzo surfaces},
        date={2002},
     journal={Trans. Amer. Math. Soc.},
      volume={354},
      number={11},
       pages={4303\ndash 4312},
}

\bib{BBJ15}{article}{
      author={Berman, Robert},
      author={Boucksom, S\'ebastien},
      author={Jonsson, Mattias},
       title={A variational approach to the {Y}au-{T}ian-{D}onaldson
  conjecture},
        date={2021},
     journal={J. Amer. Math. Soc., to appear},
        note={\url{https://doi.org/10.1090/jams/964}},
}

\bib{berman_polystability}{article}{
      author={Berman, Robert~J.},
       title={K-polystability of {${\Bbb Q}$}-{F}ano varieties admitting
  {K}\"{a}hler-{E}instein metrics},
        date={2016},
     journal={Invent. Math.},
      volume={203},
      number={3},
       pages={973\ndash 1025},
}

\bib{boyer_galicki_nakamaye}{article}{
      author={Boyer, Charles~P.},
      author={Galicki, Krzysztof},
      author={Nakamaye, Michael},
       title={On the geometry of {S}asakian-{E}instein 5-manifolds},
        date={2003},
     journal={Math. Ann.},
      volume={325},
      number={3},
       pages={485\ndash 524},
}

\bib{BHLLX20}{article}{
      author={Blum, Harold},
      author={Halpern-Leistner, Daniel},
      author={Liu, Yuchen},
      author={Xu, Chenyang},
       title={On properness of {K}-moduli spaces and optimal degenerations of
  {F}ano varieties},
        date={2021},
     journal={Selecta Math. (N.S.)},
      volume={27},
      number={4},
       pages={Paper No. 73, 39},
}

\bib{blum_jonsson}{article}{
      author={Blum, Harold},
      author={Jonsson, Mattias},
       title={Thresholds, valuations, and {K}-stability},
        date={2020},
     journal={Adv. Math.},
      volume={365},
       pages={107062, 57},
}

\bib{BLX19}{article}{
      author={Blum, Harold},
      author={Liu, Yuchen},
      author={Xu, Chenyang},
       title={Openness of {K}-semistability for {F}ano varieties},
        date={2019},
  note={\href{https://arxiv.org/abs/1907.02408}{\textsf{arXiv:1907.02408}}},
}

\bib{blum_xu_uniqueness}{article}{
      author={Blum, Harold},
      author={Xu, Chenyang},
       title={Uniqueness of {K}-polystable degenerations of {F}ano varieties},
        date={2019},
     journal={Ann. of Math. (2)},
      volume={190},
      number={2},
       pages={609\ndash 656},
}

\bib{chen_donaldson_sun}{article}{
      author={Chen, Xiuxiong},
      author={Donaldson, Simon},
      author={Sun, Song},
       title={K\"{a}hler-{E}instein metrics on {F}ano manifolds. {I}, {II},
  {III}},
        date={2015},
     journal={J. Amer. Math. Soc.},
      volume={28},
      number={1},
       pages={183\ndash 278},
}

\bib{cheltsov_lct}{article}{
      author={Cheltsov, Ivan},
       title={Log canonical thresholds of del {P}ezzo surfaces},
        date={2008},
     journal={Geom. Funct. Anal.},
      volume={18},
      number={4},
       pages={1118\ndash 1144},
}

\bib{CP18}{article}{
      author={Codogni, Giulio},
      author={Patakfalvi, Zsolt},
       title={Positivity of the {CM} line bundle for families of {K}-stable klt
  {F}ano varieties},
        date={2021},
     journal={Invent. Math.},
      volume={223},
      number={3},
       pages={811\ndash 894},
}

\bib{cheltsov_park_shramov_exceptional}{article}{
      author={Cheltsov, Ivan},
      author={Park, Jihun},
      author={Shramov, Constantin},
       title={Exceptional del {P}ezzo hypersurfaces},
        date={2010},
     journal={J. Geom. Anal.},
      volume={20},
      number={4},
       pages={787\ndash 816},
}

\bib{cheltsov_park_shramov_delta}{article}{
      author={Cheltsov, Ivan},
      author={Park, Jihun},
      author={Shramov, Constantin},
       title={Delta invariants of singular del {P}ezzo surfaces},
        date={2021},
     journal={J. Geom. Anal.},
      volume={31},
      number={3},
       pages={2354\ndash 2382},
}

\bib{del_pezzo_zoo}{article}{
      author={Cheltsov, Ivan},
      author={Shramov, Constantin},
       title={Del {P}ezzo zoo},
        date={2013},
     journal={Exp. Math.},
      volume={22},
      number={3},
       pages={313\ndash 326},
}

\bib{demailly_kollar}{article}{
      author={Demailly, Jean-Pierre},
      author={Koll\'{a}r, J\'{a}nos},
       title={Semi-continuity of complex singularity exponents and
  {K}\"{a}hler-{E}instein metrics on {F}ano orbifolds},
        date={2001},
     journal={Ann. Sci. \'{E}cole Norm. Sup. (4)},
      volume={34},
      number={4},
       pages={525\ndash 556},
}

\bib{fujita_odaka_delta}{article}{
      author={Fujita, Kento},
      author={Odaka, Yuji},
       title={On the {K}-stability of {F}ano varieties and anticanonical
  divisors},
        date={2018},
     journal={Tohoku Math. J. (2)},
      volume={70},
      number={4},
       pages={511\ndash 521},
}

\bib{Fuj19b}{article}{
      author={Fujita, Kento},
       title={K-stability of {F}ano manifolds with not small alpha invariants},
        date={2019},
     journal={J. Inst. Math. Jussieu},
      volume={18},
      number={3},
       pages={519\ndash 530},
}

\bib{Fuj19}{article}{
      author={Fujita, Kento},
       title={Uniform {K}-stability and plt blowups of log {F}ano pairs},
        date={2019},
     journal={Kyoto J. Math.},
      volume={59},
      number={2},
       pages={399\ndash 418},
}

\bib{His16}{article}{
      author={Hisamoto, Tomoyuki},
       title={{Stability and coercivity for toric polarizations}},
        date={2016},
  note={\href{https://arxiv.org/abs/1610.07998}{\textsf{arXiv:1610.07998}}},
}

\bib{Jia17}{article}{
      author={Jiang, Chen},
       title={{Boundedness of $\mathbb{Q}$-Fano varieties with degrees and
  alpha-invariants bounded from below}},
        date={2020},
     journal={Ann. Sci. \'{E}c. Norm. Sup\'{e}r. (4)},
      volume={53},
      number={4},
       pages={1235\ndash 1248},
}

\bib{kollar_johnson}{article}{
      author={Johnson, J.~M.},
      author={Koll\'{a}r, J.},
       title={K\"{a}hler-{E}instein metrics on log del {P}ezzo surfaces in
  weighted projective 3-spaces},
        date={2001},
     journal={Ann. Inst. Fourier (Grenoble)},
      volume={51},
      number={1},
       pages={69\ndash 79},
}

\bib{KKV89}{incollection}{
      author={Knop, Friedrich},
      author={Kraft, Hanspeter},
      author={Vust, Thierry},
       title={The {P}icard group of a {$G$}-variety},
        date={1989},
   booktitle={Algebraische {T}ransformationsgruppen und {I}nvariantentheorie},
      series={DMV Sem.},
      volume={13},
   publisher={Birkh\"auser, Basel},
       pages={77\ndash 87},
}

\bib{ask_petracci}{article}{
      author={Kaloghiros, Anne-Sophie},
      author={Petracci, Andrea},
       title={On toric geometry and {K}-stability of {F}ano varieties},
        date={2021},
     journal={Trans. Amer. Math. Soc. Ser. B},
      volume={8},
       pages={548\ndash 577},
}

\bib{kim_won}{article}{
      author={Kim, In-Kyun},
      author={Won, Joonyeong},
       title={Unstable singular del {P}ezzo hypersurfaces with lower index},
        date={2021},
     journal={Comm. Algebra},
      volume={49},
      number={6},
       pages={2679\ndash 2688},
}

\bib{Li19}{article}{
      author={Li, Chi},
       title={{$\mathbb{G}$-uniform stability and K\"{a}hler-Einstein metrics
  on Fano varieties}},
        date={2019},
  note={\href{https://arxiv.org/abs/1907.09399}{\textsf{arXiv:1907.09399}}},
}

\bib{LL19}{article}{
      author={Li, Chi},
      author={Liu, Yuchen},
       title={K\"{a}hler-{E}instein metrics and volume minimization},
        date={2019},
     journal={Adv. Math.},
      volume={341},
       pages={440\ndash 492},
}

\bib{ltw_1}{article}{
      author={Li, Chi},
      author={Tian, Gang},
      author={Wang, Feng},
       title={On the {Y}au-{T}ian-{D}onaldson conjecture for singular {F}ano
  varieties},
        date={2021},
     journal={Comm. Pure Appl. Math.},
      volume={74},
      number={8},
       pages={1748\ndash 1800},
}

\bib{ltw_2}{article}{
      author={Li, Chi},
      author={Tian, Gang},
      author={Wang, Feng},
       title={The uniform version of {Y}au-{T}ian-{D}onaldson conjecture for
  singular {F}ano varieties},
        date={2021},
     journal={Peking Math. J., to appear},
        note={\url{https://doi.org/10.1007/s42543-021-00039-5}},
}

\bib{LWX19}{article}{
      author={Li, Chi},
      author={Wang, Xiaowei},
      author={Xu, Chenyang},
       title={On the proper moduli spaces of smoothable {K}\"{a}hler-{E}instein
  {F}ano varieties},
        date={2019},
     journal={Duke Math. J.},
      volume={168},
      number={8},
       pages={1387\ndash 1459},
}

\bib{LWX18}{article}{
      author={Li, Chi},
      author={Wang, Xiaowei},
      author={Xu, Chenyang},
       title={Algebraicity of the metric tangent cones and equivariant
  {K}-stability},
        date={2021},
     journal={J. Amer. Math. Soc.},
      volume={34},
      number={4},
       pages={1175\ndash 1214},
}

\bib{LXZ21}{article}{
      author={Liu, Yuchen},
      author={Xu, Chenyang},
      author={Zhuang, Ziquan},
       title={{Finite generation for valuations computing stability thresholds
  and applications to K-stability}},
        date={2021},
  note={\href{https://arxiv.org/abs/2102.09405}{\textsf{arXiv:2102.09405}}},
}

\bib{LZ20}{article}{
      author={{Liu}, Yuchen},
      author={{Zhu}, Ziwen},
       title={{Equivariant K-stability under finite group action}},
        date={2020},
  note={\href{https://arxiv.org/abs/2001.10557}{\textsf{arXiv:2001.10557}}},
}

\bib{git}{book}{
      author={Mumford, D.},
      author={Fogarty, J.},
      author={Kirwan, F.},
       title={Geometric invariant theory},
     edition={Third},
      series={Ergebnisse der Mathematik und ihrer Grenzgebiete (2)},
   publisher={Springer-Verlag, Berlin},
        date={1994},
      volume={34},
        ISBN={3-540-56963-4},
}

\bib{MGS21}{article}{
      author={Martinez-Garcia, Jesus},
      author={Spotti, Cristiano},
       title={{Some observations on the dimension of Fano K-moduli}},
        date={2021},
  note={\href{https://arxiv.org/abs/2101.05643}{\textsf{arXiv:2101.05643}}},
}

\bib{MM93}{incollection}{
      author={Mabuchi, Toshiki},
      author={Mukai, Shigeru},
       title={Stability and {E}instein-{K}\"{a}hler metric of a quartic del
  {P}ezzo surface},
        date={1993},
   booktitle={Einstein metrics and {Y}ang-{M}ills connections ({S}anda, 1990)},
      series={Lecture Notes in Pure and Appl. Math.},
      volume={145},
   publisher={Dekker, New York},
       pages={133\ndash 160},
}

\bib{nadel_annals}{article}{
      author={Nadel, Alan~Michael},
       title={Multiplier ideal sheaves and {K}\"{a}hler-{E}instein metrics of
  positive scalar curvature},
        date={1990},
     journal={Ann. of Math. (2)},
      volume={132},
      number={3},
       pages={549\ndash 596},
}

\bib{odaka_sano_alpha}{article}{
      author={Odaka, Yuji},
      author={Sano, Yuji},
       title={Alpha invariant and {K}-stability of {$\Bbb Q$}-{F}ano
  varieties},
        date={2012},
     journal={Adv. Math.},
      volume={229},
      number={5},
       pages={2818\ndash 2834},
}

\bib{odaka_spotti_sun}{article}{
      author={Odaka, Yuji},
      author={Spotti, Cristiano},
      author={Sun, Song},
       title={Compact moduli spaces of del {P}ezzo surfaces and
  {K}\"{a}hler-{E}instein metrics},
        date={2016},
     journal={J. Differential Geom.},
      volume={102},
      number={1},
       pages={127\ndash 172},
}

\bib{petracci_survey}{article}{
      author={Petracci, Andrea},
       title={On deformations of toric {F}ano varieties},
        date={2019},
  note={\href{https://arxiv.org/abs/1912.01538}{\textsf{arXiv:1912.01538}}},
}

\bib{SSY16}{article}{
      author={Spotti, Cristiano},
      author={Sun, Song},
      author={Yao, Chengjian},
       title={Existence and deformations of {K}\"{a}hler-{E}instein metrics on
  smoothable {$\Bbb{Q}$}-{F}ano varieties},
        date={2016},
     journal={Duke Math. J.},
      volume={165},
      number={16},
       pages={3043\ndash 3083},
}

\bib{tian}{article}{
      author={Tian, Gang},
       title={K-stability and {K}\"{a}hler-{E}instein metrics},
        date={2015},
     journal={Comm. Pure Appl. Math.},
      volume={68},
      number={7},
       pages={1085\ndash 1156},
}

\bib{tian_alpha}{article}{
      author={Tian, Gang},
       title={On {K}\"{a}hler-{E}instein metrics on certain {K}\"{a}hler
  manifolds with {$C_1(M)>0$}},
        date={1987},
     journal={Invent. Math.},
      volume={89},
      number={2},
       pages={225\ndash 246},
}

\bib{xu_survey}{article}{
      author={Xu, Chenyang},
       title={{K}-stability of {F}ano varieties: an algebro-geometric
  approach},
        date={2020},
  note={\href{https://arxiv.org/abs/2011.10477}{\textsf{arXiv:2011.10477}}},
}

\bib{Xu19}{article}{
      author={Xu, Chenyang},
       title={A minimizing valuation is quasi-monomial},
        date={2020},
     journal={Ann. of Math. (2)},
      volume={191},
      number={3},
       pages={1003\ndash 1030},
}

\bib{XZ19}{article}{
      author={Xu, Chenyang},
      author={Zhuang, Ziquan},
       title={On positivity of the {CM} line bundle on {K}-moduli spaces},
        date={2020},
     journal={Ann. of Math. (2)},
      volume={192},
      number={3},
       pages={1005\ndash 1068},
}

\bib{XZ20}{article}{
      author={Xu, Chenyang},
      author={Zhuang, Ziquan},
       title={Uniqueness of the minimizer of the normalized volume function},
        date={2021},
     journal={Camb. J. Math.},
      volume={9},
      number={1},
       pages={149\ndash 176},
}

\bib{Zhu20}{article}{
      author={{Zhuang}, Ziquan},
       title={{Optimal destabilizing centers and equivariant K-stability}},
        date={2021},
     journal={Invent. Math., to appear},
        note={\url{https://doi.org/10.1007/s00222-021-01046-0}},
}

\end{biblist}
\end{bibdiv}

\end{document}